\documentclass[final]{siamltex}
\usepackage{cite}
\usepackage{amsmath,amssymb,algorithm,algorithmic,graphicx,enumitem}
\usepackage{url,comment}
\usepackage{paralist}
\newcommand{\real}{\mathbb{R}}

\newcommand{\rmm}{{\real}^{m \times m}}

\newcommand{\bfx}{x}

\newcommand{\bfv}{v}
\newcommand{\bfX}{\mathbf{X}}

\newcommand\textopenone{\leavevmode\hbox{\small 1\kern-3.3pt\normalsize 1}}

\newcommand{\norm}[1]{\left \lVert #1 \right \rVert}
\DeclareMathOperator{\E}{\mathbb{E}}
\DeclareMathOperator{\Prob}{\mathbb{P}}
\DeclareMathOperator{\sr}{\mathrm{sr}}
\DeclareMathOperator{\intdim}{\mathrm{intdim}}
\DeclareMathOperator{\trace}{\mathrm{trace}}
\DeclareMathOperator{\range}{\mathrm{range}}

\newtheorem{remark}[theorem]{Remark}

\title{A Probabilistic Subspace Bound\\ with Application to Active Subspaces\thanks{The second author was supported 
in part by NSF grant CCF-1145383. The second author also
acknowledges the support from the XDATA Program of the Defense Advanced
Research Projects Agency (DARPA), administered through Air Force
Research Laboratory contract FA8750-12-C-0323 FA8750-12-C-0323.  The third author was supported in part by 
the Air Force Office of Scientific Research (AROSR) grant FA9550-15-1-0299 and in part by the Consortium for Advanced Simulation of Light Water Reactors
({\tt http://www.casl.gov}), an Energy Innovation Hub ({\tt http://www.energy.gov/hubs}) for Modeling
and Simulation of Nuclear Reactors under U.S. Department of Energy Contract No.
DE-AC05-00OR22725. 
}
}
\author{
John T. Holodnak\thanks{Work conducted while a student at North Carolina State University, (\texttt{jtholodn@ncsu.edu})}
\and
Ilse C. F. Ipsen\thanks{%
Department of Mathematics, North Carolina State University, P.O. Box 8205,
Raleigh, NC 27695-8205, USA (\texttt{ipsen@ncsu.edu}, 
\texttt{http://www4.ncsu.edu/{\char'176}ipsen/})}
\and
Ralph C. Smith\thanks{%
Department of Mathematics, North Carolina State University, P.O. Box 8205,
Raleigh, NC 27695-8205, USA (\texttt{rsmith@ncsu.edu}, 
\texttt{http://www4.ncsu.edu/{\char'176}rsmith/})}
}
\begin{document}
\maketitle

\begin{abstract} 
Given a real symmetric positive semi-definite matrix $E$, and 
an approximation $S$ that is a sum of $n$ independent matrix-valued
random variables, we present bounds on the relative error in $S$ due to randomization. 
The bounds do not depend on the matrix dimensions but only on the numerical rank (intrinsic dimension) of $E$.
Our approach resembles the low-rank approximation of kernel matrices from random features, 
but our accuracy measures are more stringent.

In the context of parameter selection based on active subspaces, where  $S$ 
is computed via Monte Carlo sampling, we present a bound on the number of samples so that with high probability
the angle between the dominant subspaces of $E$ and $S$ is less than a user-specified tolerance.
This is a substantial improvement  over existing work, as it is
a non-asymptotic and fully explicit bound on the sampling amount~$n$,
and it allows the user to tune the success probability. It also suggests that Monte Carlo sampling can be efficient 
in the presence of many parameters, as long as the underlying
function $f$ is sufficiently smooth. 
\end{abstract}

\begin{keywords} 
positive semi-definite matrices, principal angles, eigenvalue decomposition, eigenvalue gaps,
matrix concentration inequality, intrinsic dimension, Monte Carlo sampling, active subspaces
\end{keywords}

\begin{AM} 
15A18, 15A23, 15A60, 15B10, 35J25, 60G60, 65N30, 65C06, 65C30, 65F15, 65D05
\end{AM}

\section{Introduction}
We analyse the accuracy of approximating a symmetric positive semi-definite matrix $E\in\rmm$
by a sum 
$\widehat{E}\equiv \tfrac{1}{n}\sum_{j=1}^n{z_jz_j^T}$ of $n$ independently sampled outer products $z_jz_j^T$,
each of which is an unbiased estimator of the mean $\E[z_jz_j^T]=E$, $1\leq j\leq n$, thus 
producing an overall unbiased estimator $\widehat{E}$.
We derive probabilistic bounds on the relative error due to randomization in $\widehat{E}$,
and on the angle between equi-dimensional dominant subspaces of $E$ and $\widehat{E}$. 
The bounds do not depend on the matrix dimension, but only on the numerical rank of $E$.

To avoid explicit dependence on the matrix dimensions, we use an intrinsic-dimension
matrix Bernstein concentration inequality. This type of analysis can be found
in low-rank approximations of kernel matrices via random features \cite{AKMMVZ17,LP14}.
However, our accuracy measures are more stringent, and the angle bounds necessitate 
assumptions that are strong enough to guarantee that the dominant subspaces are well-defined.
In contrast to existing probabilistic bounds for Krylov or subspace iterations,
see  \cite{DMKI16,HMT09} and the references therein,
our bounds are perturbation bounds and make no reference to actual methods for computing subspaces.

The motivation for this paper came from applications involving parameter selection,
such as it occurs in the solution of random ODEs and PDEs \cite{CDW2014,CWI2012,SW2014} and 
reduced-order nonlinear models \cite{BAHH2012}, as well as various applications arising in engineering, math biology, and sciences 
\cite{BNCIMN2014,CWHC2011,CDWI2011,CZC2014,NSO2015}. 

Given a function $f: \mathbb{R}^m \rightarrow \mathbb{R}$, which
depends on $m$ parameters and may be expensive to evaluate,
one wants to select subspaces associated with a few influential parameters. 
This is a form of dimension reduction \cite[Chapter 6]{Smith2014}, and one particular approach 
is to identify a low-dimensional  \textit{active subspace} in $\real^m$ 
along which $f$ is, on average, most sensitive to change \cite{CDW2014,Russi2010}. 
This is done by replacing a ``sensitivity" matrix $E\in\rmm$ by a sum of $n$ 
independent Monte Carlo samples $\widehat{E}$, followed by computing a dominant subspace of $\widehat{E}$.

Specifically,  \cite{Smith2014} illustrates that ODE models for HIV can easily have $m=20$ parameters whereas Boltzmann PDE models, quantifying neutron transport in a light water reactor, can have as many as  $m=10^6$.  In both cases, it is critical to isolate active subspaces of parameters -- which are identifiable in the sense that they are uniquely determined by observed responses --  prior to frequentist or Bayesian inference.  The matrices $E$ arise when computing local sensitivities $\nabla f(x)$ for determining these subspaces.

In the following, we present a probabilistic bound (Theorem~\ref{t_4}) that is tighter than existing work \cite{CG2014},
represents a non-asymptotic and fully explicit bound on the sampling amount~$n$,
and allows the user to tune the success probability. The absence of the matrix dimension $m$ in the bound
suggests that Monte Carlo sampling can be efficient in the presence of many parameters, 
as long as the function $f$ is sufficiently smooth.

\paragraph{Outline}
Bounds  are presented in Section~\ref{s_contrib}, with the proofs relegated to Section~\ref{s_proofs}.

\section{Our contributions}\label{s_contrib}
After stating the assumptions (Section~\ref{s_assum}), we present an upper bound for the relative error due to 
randomization in~$\widehat{E}$ (Section~\ref{s_er}),
a lower bound on the sampling amount~$n$ (Section~\ref{s_ras}), a bound on the angle between 
dominant subspaces of $E$ and $\widehat{E}$ (Section~\ref{s_sa}), and an application to
active subspaces (Section~\ref{s_as}).

\subsection{Assumptions}\label{s_assum}
Let the non-zero matrix $E\in\rmm$ be symmetric positive semi-definite,   
and $\widehat{E}\equiv \tfrac{1}{n}\sum_{j=1}^n{z_jz_j^T}$ be an approximation, where 
$z_j\in\real^m$ are $n$ independent random vectors with $\E[z_jz_j^T]=E$, $1\leq j\leq n$.
All quantities are uniformly bounded in the two norm, that is, there exists $L>0$ with 
$$\max_{1\leq j\leq n}{\|z_j\|_2}\leq L \qquad \text{and} \qquad \|E\|_2\leq L^2,$$ 
as well as 
$$\|\widehat{E}\|_2\leq \tfrac{1}{n}\,\sum_{j=1}^{n}{\|z_jz_j^T\|_2}\leq L^2.$$

Our bounds do not depend on the matrix dimension $m$, but only on the numerical rank of $E$,
which is quantified by
$$\intdim{(E)} \equiv \trace{(E)}/\|E\|_2\geq 1,$$
and equals the stable rank of $E^{1/2}$; see Section~\ref{s_matrixcon}.

\subsection{Error due to randomization}\label{s_er}
The first bound is an expression for the relative error of $\widehat{E}$ in the two norm.
\smallskip

\begin{theorem}\label{t_2} 
Given the assumptions in Section~\ref{s_assum},
for any $0<\delta<1$, with probability at least $1-\delta$, 
$$\frac{\|\widehat{E}-E\|_2}{\|E\|_2} \leq \gamma + \sqrt{\gamma(\gamma + 6)} \quad \text{where} \quad
\gamma \equiv \frac{1}{3n}\,\frac{L^2}{\|E\|_2}\, \ln\left( \frac{4}{\delta}\,\intdim{(E)} \right).$$
\end{theorem}

\begin{proof}
See Section~\ref{s_t2proof}
\end{proof}
\smallskip

Theorem~\ref{t_2} implies that the relative error in $\widehat{E}$ is small, if:
\begin{compactenum}
\item Many samples are used to compute $\widehat{E}$, that is $n\gg 1$.
\item $E$ has low numerical rank, that is, $\intdim(E)\ll m$.  
\item Tightness of the upper bound for $\|E\|_2$, that is, $L^2/\|E\|_2\approx 1$.

In the context of active subspaces in Theorem~\ref{t_4},  this is interpreted as the smoothness of an
underlying function.
\end{compactenum}

Section~\ref{s_t2expect} contains a brief discussion on alternative, expectation-based bounds. 

\subsection{Required amount of sampling}\label{s_ras}
We express Theorem~\ref{t_2} as a lower bound on the number of samples $n$ 
required for a user-specified relative error.
\smallskip

\begin{corollary}\label{c_2}
Given the assumptions in Section~\ref{s_assum}, let $0<\delta<1$ and $0<\epsilon<1$. If
$$n \geq \frac{8}{3\epsilon^2}\,\frac{L^2}{\|E\|_2} \ln\left( \frac{4}{\delta}\, \intdim{(E)}\right),$$
then with probability at least $1-\delta$,
$$\|\widehat{E}-E\|_2/\|E\|_2 \leq \epsilon.$$
\end{corollary}

\begin{proof}
See Section~\ref{s_c2proof}
\end{proof}
\smallskip

Corollary~\ref{c_2} implies that only few samples are required to compute an
approximation~$\widehat{E}$ that is highly likely to have specified accuracy $\epsilon$, if:
\begin{compactenum}
\item The requested accuracy for $\widehat{E}$ is low, that is $\epsilon \approx 1$.
\item $E$ has low numerical rank, that is $\intdim(E)\ll m$.  
\item Tightness of the upper bound for $\|E\|_2$, that is, $L^2/\|E\|_2\approx 1$.
\end{compactenum}

\subsection{Subspace bound}\label{s_sa}
We bound the largest principal angle between equi-dimensional dominant subspaces of
of $E$ and $\widehat{E}$. 
To identify the subspaces, consider the eigenvalue decompositions
\begin{eqnarray}\label{e_ev1}
E = V\Lambda V^T, \qquad
\Lambda = \diag\begin{pmatrix}\lambda_1 & \cdots& \lambda_m\end{pmatrix}, \qquad
V=\begin{pmatrix}v_1 &\ldots & v_m\end{pmatrix},
\end{eqnarray}
where $\lambda_1 \geq \cdots \geq \lambda_m\geq 0$ and $V\in\rmm$ is an orthogonal matrix,
and
\begin{eqnarray}\label{e_ev2}
\widehat{E} = \widehat{V}\widehat{\Lambda}\widehat{V}^T, \qquad
\widehat{\Lambda}=\begin{pmatrix}\widehat{\lambda}_1 & \cdots & \widehat{\lambda}_m\end{pmatrix},\qquad 
\widehat{V}=\begin{pmatrix}\widehat{v}_1 &\cdots & \widehat{v}_m\end{pmatrix},
\end{eqnarray}
where  $\widehat{\lambda}_1 \geq \cdots \geq \widehat{\lambda}_m\geq 0$,
and $\widehat{V}\in \rmm$ is an orthogonal matrix.
 
The bound below on the largest principal angle $\angle(\widehat{\mathcal{S}},\mathcal{S})$ 
between dominant subspaces $\mathcal{S}$ and $\widehat{\mathcal{S}}$
requires the perturbation to be sufficiently small compared to the eigenvalue gap.
\smallskip

\begin{theorem}\label{t_3}
In addition to the assumptions in Section~\ref{s_assum}, 
let  $E$ have an eigenvalue gap $\lambda_{k} - \lambda_{k+1} > 0$  for some $1\leq k<m$ , 
so that
$\mathcal{S}\equiv \range\begin{pmatrix}v_1 & \cdots & v_k\end{pmatrix}$ is well-defined.
Also, let $0<\epsilon < \frac{\lambda_k - \lambda_{k+1}}{4\,\|E\|_2}$ and $0<\delta<1$. 

 If the sampling amount is sufficiently large,
$$n \geq \frac{8}{3\,\epsilon^2}\,\frac{L^2}{\|E\|_2}\, \ln\left( \frac{4}{\delta}\,\intdim{(E)} \right),$$
then with probability at least $1-\delta$, the dominant subspace 
$\widehat{\mathcal{S}}=\range\begin{pmatrix}\widehat{v}_1&\cdots & \widehat{v}_k\end{pmatrix}$ 
is well-defined,
and the largest principal angle $\angle(\widehat{\mathcal{S}},\mathcal{S})$ is bounded by
$$\sin{\angle(\widehat{\mathcal{S}},\mathcal{S})}\leq 
\frac{4\,\|E\|_2}{\lambda_k - \lambda_{k+1}}\, \epsilon.$$
\end{theorem}

\begin{proof}
See Section~\ref{s_t3proof}.
\end{proof}
\smallskip

Theorem~\ref{t_3} implies that $\widehat{\mathcal{S}}$ is likely to be $\epsilon$-close to $\mathcal{S}$, if:
\begin{compactenum}
\item $E$ has a large eigenvalue gap $(\lambda_k-\lambda_{k+1})/\lambda_1$.
This is a relative gap, but it is weak because the denominator contains $\lambda_1$ rather than $\lambda_k$.
The inverse of the  gap is a measure of sensitivity for the subspace~$\mathcal{S}$. 
\item The matrix $E$ has low numerical rank,  that is $\intdim(E)\ll m$.
\item Tightness of the upper bound for $\|E\|_2$, that is, $L^2/\|E\|_2\approx 1$.
\end{compactenum}

\subsection{Application to active subspaces}\label{s_as}
After setting the context (Section~\ref{s_setting}) we improve an existing bound 
on the number of Monte Carlo samples required  to approximate an active subspace to
a user-specified error subject to a user-specified success probability (Section~\ref{s_improv}).

\subsubsection{Problem Setting}\label{s_setting}
Assume the non-constant
function $f(\bfx): \real^m \rightarrow \real$ is continuously differentiable, with gradient vector 
$$\nabla f(\bfx) =\begin{pmatrix} \frac{\partial f}{\partial \bfx_1}(x) & \cdots & 
\frac{\partial f}{\partial \bfx_m}(x)\end{pmatrix}^T\in \real^m$$ 
and Lipschitz constant $L > 0$ so that $\|\nabla f(\bfx)\|_2 \leq L$ for all $\bfx \in \real^m$.
Assume also that $f$ is square integrable on $\real^m$ with respect to a positive and bounded probability density 
function $\rho(\bfx)$; and that all products of partial derivatives of $f$ are also integrable with respect to $\rho(\bfx)$.

Let $\bfX\in\real^m$ be a random vector with the associated probability density function $\rho(\bfx)$, and denote by 
$$\E{[h(\bfX)]} \equiv \int_{\real^m}{h(\bfx)\,\rho(\bfx)\, d\bfx}$$
the expected value of a function $h(\bfX):\real^m\rightarrow \real$ with regard to $\bfX$.
Then the sensitivity of $f$ along a unit-norm direction $\bfv$ can be estimated from  
the expected value of the squared directional derivative of $f$ along~$\bfv$, 
\begin{eqnarray*}
\E{[\left(\bfv^T\,\nabla f(\bfX)\right)^2]} = \int_{\real^m}{\left(\bfv^T\,\nabla f(\bfx) \right)^2\rho(\bfx)\, d\bfx}.
\end{eqnarray*}
Directional derivatives \cite{AMS2008} can measure sensitivity in any direction, while
mean squared derivatives \cite{Sobol2009} are limited to coordinate directions. 
Informative directional derivatives can be obtained  from the $m\times m$ matrix \cite[Lemma 2.1]{CDW2014}
\begin{eqnarray}\label{e_E}
E \equiv \E\left[\nabla f(\bfX) \left( \nabla f(\bfX)\right)^T \right]=
\int_{\real^m}{ \nabla f(\bfx) \left( \nabla f(\bfx)\right)^T \rho(\bfx)\,d\bfx}
\end{eqnarray}
and its eigenvalue decomposition (\ref{e_ev1}). For an eigenpair $(\lambda_j, v_j)$, with  $Ev_j=\lambda_j v_j$,
we have $\lambda_j=\E{[(\bfv_j^T\,\nabla f(\bfX))^2]}$.
This means that eigenvector~$v_j$ indicates a direction of sensitivity for $f$, 
while $\lambda_j$ represents the average amount of sensitivity of~$f$ along~$v_j$. 
In particular, eigenvectors associated with the dominant eigenvalues represent directions along 
which $f$, on average, is the most sensitive. 

This leads to the concept  of active subspace \cite{CDW2014,erratum,CG2014}.
We assume that  for some $1\leq k<m$, the matrix $E$ has an eigenvalue gap $\lambda_k > \lambda_{k+1}$.
Then the dominant subspace
$\mathcal{S}=\range\begin{pmatrix}v_1 & \cdots &v_k\end{pmatrix}$ 
is well-defined and called the \textit{active subspace of dimension~$k$ of~$f$}.

Since explicit computation of $E$ in (\ref{e_E}) is often not feasible because  
the elements are high-dimensional integrals, one can use a Monte Carlo method \cite[(2.16)]{CDW2014} 
to independently sample $n$ vectors $\bfx_j \in \real^m$ according to $\rho(\bfx)$, 
and approximate $E$ by 
$$\widehat{E} = \frac{1}{n} \sum_{j=1}^{n}{\nabla f(\bfx_j) \left( \nabla f(\bfx_j)\right)^T}.$$
If  $\widehat{E}$ happens to have an eigenvalue gap at the same location as $E$, so that 
$\widehat{\lambda}_k > \widehat{\lambda}_{k+1}$, then the perturbed dominant subspace 
$\widehat{\mathcal{S}}=\range{\begin{pmatrix}\widehat{v}_1 & \cdots & \widehat{v}_k\end{pmatrix}}$ 
is also well-defined and called \textit{approximate active subspace} of dimension~$k$ for~$f$.

\subsubsection{Accuracy of approximate active subspace}\label{s_improv}
The bounds below are conceptual as they depend on unknown quantities like the eigenvalues of~$E$
and a global bound on the gradient norm. Nevertheless, a sufficiently tight bound is informative because it
suggests that Monte Carlo sampling can be efficient in the presence of many parameters, as long as 
the function $f$ is sufficiently smooth with most eigenvalues of $E$ being small. 

Below we tried to collect all the required assumptions for \cite[Corollary 3.7]{CG2014}.
\smallskip
 
\begin{theorem}[Theorem 3.5, Corollaries 3.6 and  3.7 in \cite{CG2014}]\label{t_5}
In addition to the assumptions in Section~\ref{s_setting}, also assume that 
$0<\epsilon\leq \min\left\{1,\tfrac{\lambda_k-\lambda_{k+1}}{5\,\lambda_1}\right\}$ and
$$\nu\equiv \left\|\int{\left((\nabla f(x) (\nabla f(x))^T-E\right)^2\rho(x)\, dx}\right\|>0.$$
If the number of samples  is 
$$n=\Omega\left(\max\left\{\frac{L^2}{\lambda_1\,\epsilon},\, 
\frac{\nu}{\lambda_1^2\,\epsilon^2}\right\}\> \ln( 2\, m)\right),$$
then with high probability 
$$\sin{\angle(\widehat{\mathcal{S}},\mathcal{S})}\leq
\frac{4\,\lambda_1\,\epsilon}{\lambda_k - \lambda_{k+1}}.$$
\end{theorem}

We  improve on Theorem~\ref{t_5} by presenting a bound that is tighter and more informative. More specifically,
Theorem~\ref{t_4}
(i) specifies a non-asymptotic, fully explicit, computable bound for the sampling amount~$n$;
(ii) specifies an explicit expression for the failure probability~$\delta$ and its 
impact on the sampling amount~$n$, thus allowing tuning by the user;
(iii) depends on the numerical rank of $E$, which can be much smaller than the total number $m$ of parameters; and
(iv) guarantees that the approximate active subspace $\widehat{\mathcal{S}}$ is well-defined.
\smallskip

\begin{theorem}\label{t_4}
With the assumptions in Section~\ref{s_setting}, let
$0<\epsilon < \tfrac{\lambda_k - \lambda_{k+1}}{4\,\|E\|_2}$ and $0<\delta<1$. 
If the number of Monte Carlo samples is at least
$$n \geq \frac{8}{3\,\epsilon^2}\,\frac{L^2}{\|E\|_2}\, \ln\left( \frac{4}{\delta}\,\intdim{(E)} \right),$$
then with probability at least $1-\delta$,
the approximate active subspace $\widehat{\mathcal{S}}$ is well-defined, 
and the largest principal angle $\angle(\widehat{\mathcal{S}},\mathcal{S})$ is bounded by
$$\sin{\angle(\widehat{\mathcal{S}},\mathcal{S})}\leq 
\frac{4\,\|E\|_2}{\lambda_k - \lambda_{k+1}}\, \epsilon.$$
\end{theorem}

\begin{proof}
See Section~\ref{s_t4proof}.
\end{proof}
\smallskip

Theorem~\ref{t_4} implies that $\widehat{\mathcal{S}}$ is likely to be an $\epsilon$-accurate  
approximation to the active subspace $\mathcal{S}$, if:
\begin{compactenum}
\item $E$ has a large relative eigenvalue gap $(\lambda_k-\lambda_{k+1})/\lambda_1$.
\item The matrix $E$ has low numerical rank,  that is $\intdim(E)\ll m$.
\item The function $f$ is smooth, in the sense that $L^2/\|E\|_2\approx 1$, see Section~\ref{s_t4proof}.
\end{compactenum}
\smallskip

\begin{remark}
Monte Carlo sampling of $x_j$, according to $\rho(x)$, does not necessarily produce gradients
$\nabla f(\bfx_j) \left( \nabla f(\bfx_j)\right)^T$ that concentrate tightly around the mean $E$.
This could be remedied with some form of  importance sampling.

For instance, in order to speed up the sampling of Fourier features for kernel ridge regression, 
\cite[Section 4]{AKMMVZ17} propose to sample according to the leverage function of the kernel.
However, it is not clear how this can be implemented efficiently in practice; plus the required sampling
amount appears to exhibit a much stronger dependence on the
problem dimension than is acceptable in our context. 
\end{remark}

\section{Proofs}\label{s_proofs}
We present all the materials required for the proofs of Theorem~\ref{t_2} (Section~\ref{s_t2proof}),
Corollary~\ref{c_2} (Section~\ref{s_c2proof}), Theorem~\ref{t_3} (Section~\ref{s_t3proof}),
and Theorem~\ref{t_4} (Section~\ref{s_t4proof}).

\subsection{Everything for the proof of Theorem~\ref{t_2}}\label{s_t2proof}
The idea is to view $\widehat{E}$ as a sum of random variables.
To this end we review a matrix Bernstein concentration inequality and the definition of 
intrinsic dimension (Section~\ref{s_matrixcon}), and then apply the concentration inequality 
to prove Theorem~\ref{t_2} (Section~\ref{s_t2p}), followed by a brief discussion of expectation-based
bounds (Section~\ref{s_t2expect}).

\subsubsection{Matrix Bernstein concentration inequality, and intrinsic dimension}\label{s_matrixcon}
Concentration inequalities bound the deviation of a sum of random variables from the mean.

Here, the random variables are matrix-valued and bounded; and
have zero mean, and bounded ``variance"  in the sense of the
 L\"owner partial order\footnote{If $P_1$ and $P_2$ are real symmetric matrices, then 
$P_1\preceq P_2$ means that $P_2-P_1$ is positive semi-definite \cite[Definition 7.7.1]{HJ2013}.}.
We use a matrix Bernstein concentration inequality with intrinsic dimension  \cite[Section 7.2]{Tropp2015}
in the context of a random sampling model \cite[page 83]{Tropp2015}. 
The \textit{intrinsic dimension} quantifies the numerical rank of a symmetric positive semi-definite matrix, and 
is instrumental in avoiding an explicit dependence on the matrix dimension.

\begin{definition}[Section 2.1 in \cite{Mins17}, Definition 7.1.1 in \cite{Tropp2015}]\label{d_intdim}
The {\rm intrinsic dimension} or {\rm effective rank} of a non-zero, symmetric positive semi-definite 
matrix $P\in\rmm$ is  $\intdim{(P)} \equiv \trace{(P)}/\|P\|_2$.
\end{definition}

The symmetric positive semi-definiteness of $P$ implies that its \textit{intrinsic dimension} is bounded by the rank
$$1\leq \intdim{(P)} \leq \rank{(P)}\leq m,$$
and is equal to the \textit{stable rank} of a square root \cite[Section 6.5.4]{Tropp2015}
\begin{eqnarray*}
\intdim(P) & = &\frac{\lambda_1(P)+\cdots +\lambda_m(P)}{\lambda_1(P)}=
\frac{\sigma_1(P^{1/2})^2+\cdots+\sigma_m(P^{1/2})^2}{\sigma_1(P^{1/2})}\nonumber\\
&=&\left(  \frac{\|P^{1/2}\|_F}{\|P^{1/2}\|_2}\right)^2=\sr{(P^{1/2})}.
\end{eqnarray*}
\smallskip

\begin{theorem}[\cite{MIns11,Mins17} and Theorem 7.3.1 in \cite{Tropp2015}]\label{t_mins}
If
\begin{compactenum}
\item (Independence) $\ X_j$ are $n$ independent real symmetric random matrices,
\item (Boundedness) $\ \max_{1\leq j\leq n}{\|X_j\|_2}\leq \beta$ for some $\beta>0$,
\item (Zero mean)  $\  \E{[X_j]}=0$, $1\leq j\leq n$,
\item  (Bounded matrix variance) $\ \sum_{j=1}^{n}{\E{[X_j^2]}}\preceq P$ for some~$P$,
\item (Sufficient  tolerance) $\epsilon \geq \norm{P}_2^{1/2} + \beta/3$, 
\end{compactenum}
then
$$\Prob{\left[\left\|\sum_{j=1}^{n}{X_j}\right\|_2\geq \epsilon \right]} \leq 
4 \>\intdim{(P)}\>\exp\left( \frac{-\epsilon^2/2}{\norm{P}_2 + \beta\epsilon/3} \right).$$
\end{theorem}
\smallskip

Since $\E{\left[\sum_{j=1}^{n}{X_j}\right]}=0$, Theorem~\ref{t_mins} is a bound 
for the deviation of the sum from its mean, and implies that a large deviation  is unlikely
if the matrix variance has low rank. Most importantly, the bound does not depend on 
the dimension of the matrices~$X_j$.

\subsubsection{Proof of Theorem~\ref{t_2}}\label{s_t2p}
The proof is inspired by \cite[Theorem 7.8]{HI2015} and similar in part to the proof of \cite[Theorem 3]{LP14}.
Set
$$X_j\equiv z_jz_j^T \qquad \text{and} \qquad Y_j\equiv \frac{1}{n}\left(X_j-E\right), \qquad 1\leq j\leq n,$$
so that $\sum_{j=1}^n{Y_j}=\sum_{j=1}^n{X_j}-E$.
Before applying Theorem \ref{t_mins} to the sum of the $Y_j$, we need to verify 
that they satisfy the assumptions. 

\paragraph{Independence}
By assumption, the $X_j$ are independent, and so are the $Y_j$.

\paragraph{Zero mean}
Since $\E[X_j]=E$, the linearity of the expected value implies
$$\E{[Y_j]}= \tfrac{1}{{n}}\left(\E[X_j] - E\right)=\tfrac{1}{n}(E-E)=0.$$ 

\paragraph{Boundedness} 
The positive semi-definiteness and boundedness of $X_j$ and $E$ from Section~\ref{s_assum} imply
\begin{eqnarray}\label{e_p1}
\|Y_j\|_2 \leq \frac{1}{n}\,\max\left\{ \|X_j\|_2, \|E\|_2\right\} \leq \frac{1}{n}\max \{L^2,\|E\|_2 \}
\leq \beta \equiv \frac{L^2}{n}.
\end{eqnarray}

\paragraph{Matrix variance}
Multiply out, and apply the definition of $E$, 
\begin{eqnarray*}
\E{[Y_j^2]} & = & \frac{1}{n^2}\E{\left[(X_j  -  E)^2\right]}  
 =   \frac{1}{n^2}\left( \E[X_j^2] -E\,\E[X_j] - \E[X_j]\,E + E^2\right) \\
 & = & \frac{1}{n^2}\left( \E[X_j^2] - E^2 \right).
 \end{eqnarray*}
Since $E$ is positive semi-definite, we can drop it without decreasing semi-definiteness,
\begin{eqnarray}\label{e_var}
\E{[Y_j^2]} \preceq  \frac{1}{n^2} \E{[X_j^2]}.
\end{eqnarray}

\paragraph{Bounded matrix variance}  
Since $X_j$ is an outer product, 
$$X_j^2=z_j^Tz_j \, X_j = \|z_j\|_2^2\, X_j\preceq L^2 X_j,$$ 
thus $\E[X_j^2] \preceq L^2\E[X_j]= L^2E$.
This, together with (\ref{e_var}) gives $\E{\left[Y_j^2\right]} \preceq \frac{L^2}{n^2}\,E$. 
The linearity of the expected value across the $n$ identically distributed summands implies
\begin{eqnarray}\label{e_p2}
\sum_{j=1}^{n}{\E{[Y_j^2]}}\preceq \sum_{j=1}^n{\frac{L^2}{n^2} E}= P\equiv \frac{L^2}{n}\,E,
\end{eqnarray}
where $P$ is symmetric positive semi-definite since $E$ is.

From $\|P\|_2 = \tfrac{L^2}{n}\,\|E\|_2$ and $\trace{(P)}=\tfrac{L^2}{n}\trace{(E)}$ follows  
\begin{eqnarray*}\label{e_P}
\intdim{(P)} = \trace{(E)}/\|E\|_2.
\end{eqnarray*}

\paragraph{Application of Theorem \ref{t_mins}}
Substituting (\ref{e_p1}) and (\ref{e_p2}) into Theorem \ref{t_mins} gives
the probability for the absolute error
$$\Prob{\left[\|S-E\|_2\geq \hat{\epsilon} \right]} \leq 4 \>\intdim{(E)}\>
 \exp\left( - \frac{n}{L^2}\,\frac{\hat{\epsilon}^2/2}{\|E\|_2 + \hat{\epsilon}/3} \right).$$
Setting $\hat{\epsilon}=\|E\|_2\, \epsilon$ gives the probability for the relative error
$$\Prob{\left[\frac{\|S-E\|_2}{\|E\|_2}\geq \epsilon \right]} \leq 4 \>\intdim{(E)}\>
 \exp\left( - n \,\frac{\|E\|_2}{L^2}\,\frac{\epsilon^2/2}{1 + \epsilon/3} \right).$$
Setting the above right hand side equal to $\delta$ and solving for $\epsilon$ gives 
$$\epsilon = \gamma + \sqrt{\gamma\,(\gamma + 6)} \qquad
\text{where} \quad
\gamma =\frac{1}{3n}\,\frac{L^2}{\|E\|_2}\,\ln\left(\frac{4}{\delta}\, \intdim{(E)} \right).$$

\paragraph{Sufficient tolerance}
We still need to check the lower bound for $\hat{\epsilon}$ and verify that  $\hat{\epsilon} \geq \beta/3+\|P\|_2^{1/2}$,
which is equivalent to  verifying that
$$\epsilon \geq \frac{\beta}{3\|E\|_2}+\frac{\|P\|_2^{1/2}}{\|E\|_2}.$$
From $0<\delta<1$ and  $\intdim{(E)}\geq 1$ follows
$e<\tfrac{4}{\delta}\leq \tfrac{4}{\delta}\,\intdim{(E)}$, which implies
$\ln(\frac{4}{\delta}\,\intdim{(E)})\geq 1$. Together with (\ref{e_p1}) this gives
$$\frac{\beta}{3\|E\|_2}=\frac{1}{3n}\, \frac{L^2}{\|E\|_2}
\leq \frac{1}{3n}\,\frac{L^2}{\|E\|_2}\>\ln\left(\frac{4}{\delta}\, \intdim{(E)} \right)=\gamma,$$
and with (\ref{e_p2})
$$\frac{\|P\|_2^{1/2}}{\|E\|_2}=\frac{1}{\|E\|_2}\,\sqrt{\frac{L^2}{n}\|E\|_2}=
\sqrt{\frac{1}{n}\,\frac{L^2}{\|E\|_2}}
\leq \sqrt{6\gamma}\leq \sqrt{\gamma\,(\gamma + 6)}.$$
Adding the two previous inequalities gives the required lower bound 
\begin{eqnarray}\label{e_epsilon}
\epsilon \geq \gamma+ \sqrt{\gamma\,(\gamma + 6)}\geq \frac{\beta}{3\|E\|_2}+\frac{\|P\|_2^{1/2}}{\|E\|_2}.
\end{eqnarray}

\subsubsection{Expectation-based bounds}\label{s_t2expect}
Following the observation  \cite[Section 4.1]{Tropp2015} that matrix concentration bounds
may not always give satisfactory information about the tail, we consider
an alternative to the exponential Bernstein concentration inequality 
 \cite[Theorem 7.3.1]{Tropp2015} represented here by Theorem~\ref{t_mins},
and combine an expectation bound with the scalar Markov inequality.

The intrinsic dimension expectation bound \cite[Corollary 7.3.2]{Tropp2015}, 
together with the assumptions in Theorem~\ref{t_mins} and the bounds in \cite[Section 7.4.4]{Tropp2015}
implies for matrices~$P$ with $\intdim(P)\geq 2$ that 
$$\E\left[\left\|\sum_{j=1}^n{X_j}\right\|_2\right] \leq
 \frac{10}{3} \left( \sqrt{\|P\|_2\,\theta}+ \beta\,\theta\right), \qquad \theta\equiv \ln{(1+\intdim(P))}.$$ 
Combined with Markov's inequality \cite[(2.2.1)]{Tropp2015} this gives
$$\delta\equiv \Prob\left[\|\widehat{E}-E\|_2\geq \epsilon\right]\leq 
\E\left[\|\sum_{j=1}^n{X_j}\|_2\right]/\epsilon.$$
Thus, the error bound is inversely proportional to the failure probability, 
$$\epsilon=\mathcal{O}(1/\delta).$$  
In contrast, Theorem~\ref{t_2} implies the much weaker dependence,
$$\epsilon=\mathcal{O}(\ln{(1/\delta})).$$  
Since we are interested in extremely small failure probabilities, on the order of machine epsilon,
 $\delta=10^{-15}$, 
exponential concentration inequalities are preferable to expectation-based bounds in our context.

\subsection{Proof of Corollary~\ref{c_2}}\label{s_c2proof}
A lower bound on $\epsilon$ is required by the last assumption of Theorem~\ref{t_mins} 
and was established in (\ref{e_epsilon}), which is equivalent to $2\,\gamma\,(3+\epsilon)\leq \epsilon^2$.
This bound and therefore (\ref{e_epsilon})  definitely holds if 
\begin{eqnarray}\label{e_gamma}
\gamma\leq \epsilon^2/8
\end{eqnarray} 
because $\epsilon<1$. Setting $\gamma=\alpha/n$ and
$\alpha\equiv \tfrac{L^2}{3\,\|E\|_2}\,\ln\left(\frac{4}{\delta}\, \intdim{(E)}\right)$
shows that (\ref{e_gamma}) is equivalent to the desired lower bound for $n$.

\subsection{Everything for the proof of Theorem~\ref{t_3}}\label{s_t3proof}
We verify the conditions and apply a deterministic bound for the subspace angle (Section~\ref{s_structural}),
and then present the proof of Theorem~\ref{t_3} (Section~\ref{s_t3p}).

\subsubsection{Deterministic subspace angle bound}\label{s_structural}
We bound $\sin{\angle(\widehat{\mathcal{S}},\mathcal{S})}$ in terms of the absolute error $\|\widehat{E}-E\|_2$.

The keep the notation simple, partition the eigenvectors in  (\ref{e_ev1}) 
\begin{eqnarray}\label{e_part}
V = \begin{pmatrix} V_1 & V_2\end{pmatrix} \qquad \text{where}\qquad
V_1\equiv\begin{pmatrix}v_1 &\cdots & v_k\end{pmatrix}\in\real^{m\times k},
\end{eqnarray}
and conformally partition the eigenvectors in (\ref{e_ev2}), 
\begin{eqnarray*}
\widehat{V} = \begin{pmatrix} \widehat{V}_1 & \widehat{V}_2\end{pmatrix},
\qquad \text{where} \qquad
\widehat{V}_1\equiv \begin{pmatrix}\widehat{v}_1 &\cdots & \widehat{v}_k\end{pmatrix}\in\real^{m\times k}.
\end{eqnarray*}

Next is a straightforward specialization of \cite[Theorems 2.7 and 4.11]{Ste73}, \cite[Theorem V.2.7]{StS90},
and \cite[Corollary 8.1.11]{GovL2013} to real symmetric matrices and the two norm.
\smallskip

\begin{lemma}\label{l_1}  
Partition as in (\ref{e_part}), 
$$F=\begin{pmatrix}F_{11} & F_{12} \\ F_{12}^T &F_{22}\end{pmatrix}\equiv 
\begin{pmatrix}V_1 & V_2\end{pmatrix}^T\,(\widehat{E}-E)\,\begin{pmatrix}V_1 & V_2\end{pmatrix}.$$
If $\mathrm{gap}\equiv \lambda_k-\lambda_{k+1}>0$,
\begin{eqnarray}\label{e_aux1}
\eta \equiv \mathrm{gap}- \|F_{11}\|_2-\|F_{22}\|_2>0 \qquad  \text{and} \qquad
 \frac{\|F_{12}\|_2}{\mathrm{gap}}<\tfrac{1}{2}
\end{eqnarray}
then 
$$\|V_1V_1^T-\widehat{V}_1\widehat{V}_1^T\|_2\leq 2\,\|F_{12}\|_2/\eta.$$
\end{lemma}

Now comes the deterministic basis for Theorem~\ref{t_3}, and it
requires the perturbation $\|\widehat{E}-E\|_2$ to be sufficiently small compared to the eigenvalue gap.
The conclusions are spelled out in more detail than
usual to ensure a correct interface with the matrix concentration bounds for Section~\ref{s_t3p}.  
\smallskip

\begin{theorem}\label{t_1}
If for some $1\leq k<m$, the matrix $E$ has an eigenvalue gap
$\lambda_{k} - \lambda_{k+1} > 0$ and $\|\widehat{E}-E\|_2 < (\lambda_{k}-\lambda_{k+1})/4$, then
\begin{compactenum}
\item $\widehat{E}$ has an eigenvalue gap at the same location as $E$, that is,
$\widehat{\lambda}_k-\widehat{\lambda}_{k+1}>0$.
\item The dominant subspaces $\mathcal{S}=\range{(V_1})$ 
and $\widehat{\mathcal{S}}=\range{(\widehat{V}_1)}$  are well-defined.
\item The largest principal angle $\angle(\widehat{\mathcal{S}},\mathcal{S})$ is bounded by
$$\sin{\angle(\widehat{\mathcal{S}},\mathcal{S})} \leq 4\,\frac{\|\widehat{E}-E\|_2}{\lambda_k-\lambda_{k+1}}.$$
\end{compactenum}
\end{theorem}

\begin{proof}
With the abbreviations $\tau\equiv \|\widehat{E}-E\|_2$ and $\mathrm{gap}\equiv \lambda_k-\lambda_{k+1}>0$,
the all important assumption takes the form
\begin{eqnarray}\label{e_aux4}
\|F\|_2=\tau < \mathrm{gap}/4.
\end{eqnarray}
The three statements will now be proved in the order listed. 

\begin{compactenum}
\item To show that $\widehat{\lambda}_k-\widehat{\lambda}_{k+1}>0$, invoke
the Cauchy interlace Theorem \cite[Section 10-1]{Par80}, 
$\max_{1\leq j\leq m}{|\lambda_j-\widehat{\lambda}_j|}\leq \tau$, which implies in particular
$\widehat{\lambda}_k\geq \lambda_k-\tau$ and $\widehat{\lambda}_{k+1}\leq \lambda_{k+1}+\tau$.
Together with (\ref{e_aux4}) this gives
$$\widehat{\lambda}_k-\widehat{\lambda}_{k+1}\geq (\lambda_k-\tau)-(\lambda_{k+1}+\tau)
=\mathrm{gap}-2\tau> 4\tau-2\tau=2\tau\geq 0.$$

\item The existence of the eigenvalue gaps $\lambda_k-\lambda_{k+1}>0$ and
$\widehat{\lambda}_k-\widehat{\lambda}_{k+1}>0$ implies that $\mathcal{S}$ and 
$\widehat{\mathcal{S}}$ are simple invariant subspaces \cite[Section V.1]{StS90}.

Since the columns of $V_1,\widehat{V}_1\in\real^{m\times k}$ are orthonormal bases for  
$\mathcal{S}=\range(V_1)$ and $\widehat{\mathcal{S}}=\range(\widehat{V}_1)$, respectively, 
the matrices $V_1V_1^T$ and $\widehat{V}_1\widehat{V}_1^T$ are orthogonal 
projectors onto $\mathcal{S}$ and $\widehat{\mathcal{S}}$, respectively.

\item The two norm difference between orthogonal  projectors onto equi-dimensional spaces
is the sine of the largest principal angle 
 \cite[Sections 2.5.3, 6.4.3]{GovL2013}, \cite[Corollary 2.6]{Ste73}, 
\begin{eqnarray}\label{e_aux3}
\sin{\angle(\widehat{\mathcal{S}},\mathcal{S})} =\|V_1V_1^T - \widehat{V}_1\widehat{V}_1^T\|_2.
\end{eqnarray}
To bound this difference in terms of $\tau$,  we apply Lemma~\ref{l_1}, but need to verify first
that its conditions (\ref{e_aux1}) hold.

The first condition in (\ref{e_aux1}) follows from (\ref{e_aux4}) and
\begin{eqnarray*}
\eta &=& \mathrm{gap}- \|F_{11}\|_2-\|F_{22}\|_2 \geq \mathrm{gap}-2\|F\|_2
\geq \mathrm{gap}-\tfrac{1}{2}\mathrm{gap}=\tfrac{1}{2}\mathrm{gap}>0.
\end{eqnarray*}
The second condition in (\ref{e_aux1}) also follows from (\ref{e_aux4}) and
$$\frac{\|F_{12}\|_2}{\mathrm{gap}}\leq \frac{\|F\|_2}{\mathrm{gap}} <\tfrac{1}{4}<\tfrac{1}{2}.$$
The desired bound follows from combining (\ref{e_aux3}), Lemma~\ref{l_1} and $\eta\geq \tfrac{1}{2}\mathrm{gap}$.
\end{compactenum}
\end{proof}
\smallskip

Theorem~\ref{t_1} implies that the subspace $\mathcal{S}$ is well-conditioned if the eigenvalue gap
$\lambda_k-\lambda_{k+1}$ is large compared to the matrix perturbation $\|\widehat{E}-E\|_2$.

\subsubsection{Proof of Theorem~\ref{t_3}}\label{s_t3p}
We combine the probabilistic bound in Corollary~\ref{c_2} and the deterministic bound in Theorem~\ref{t_1}. 

Corollary~\ref{c_2} implies: If
\begin{eqnarray}\label{e_nb}
n \geq \frac{8}{3\,\epsilon^2}\,\frac{L^2}{\|E\|_2} \ln\left( \frac{4}{\delta}\, \intdim{(E)}\right),
\end{eqnarray}
then with probability at least $1-\delta$ we have
$\|\widehat{E}-E\|_2\leq \|E\|_2 \,\epsilon$.
Theorem~\ref{t_3} guarantees, by assumption, that $\|E\|_2\,\epsilon< (\lambda_k-\lambda_{k+1})/4$.

Combining the two gives: If (\ref{e_nb}) holds, then with probability at least $1-\delta$ we have
$\|\widehat{E}-E\|_2\leq (\lambda_k-\lambda_{k+1})/4$.
This in turn means: If (\ref{e_nb}) holds then with probability at least $1-\delta$ the assumptions
for Theorem \ref{t_1} are satisfied, and its conclusions hold.  
\smallskip

\begin{remark}
We do not see how to transfer the eigenvalue gap $\lambda_k-\lambda_{k+1}$ from the angle bound
into the number of samples. 

This is  because the deterministic bound in Theorem~\ref{t_1} and the probabilistic bound in Corollary~\ref{c_2}
make competing demands. The former requires an upper bound on the matrix perturbation $\|\widehat{E}-E\|_2$,
while the latter requires a lower bound.
\end{remark}

\subsection{Proof of Theorem~\ref{t_4}}\label{s_t4proof}
Once the required bound for $E$ has been established below, the proof of Theorem~\ref{t_4} is
is a direct consequence of Theorem~\ref{t_3}.
\smallskip

\begin{lemma}\label{l_ebound}
The matrix $E$ in Theorem~\ref{t_4} satisfies $\|E\|_2\leq L^2$. 
\end{lemma}
\smallskip

\begin{proof}
This follows from the assumptions in Section~\ref{s_as} and the fact that $\rho$ is a probability density function,
\begin{eqnarray*}
\|E\|_2 &=& \left\| \int{\nabla f(\bfx) (\nabla f(\bfx))^T  \rho(\bfx)\, d\bfx}\right\|_2 
\leq  \int{\|\nabla f(\bfx)(\nabla f(\bfx))^T\|_2\,\rho(\bfx)\, d\bfx} \\
&\leq & \max_{x}{\{\|\nabla f(\bfx)\|_2^2\}} \, \int\rho(\bfx)\,d\bfx \leq  L^2.
\end{eqnarray*}
\end{proof}

Thus we can interpret 
\begin{eqnarray}\label{e_ebound}
L^2/\|E\|_2\geq 1
\end{eqnarray} 
as a measure for the smoothness of $f$.

\section{Acknowledgements}
We thank Haim Avron, David Bindel, Serkan Gugercin, Mert Gurbuzbalaban, Tim Kelley, 
and Jim Nagy for helpful discussions.

\bibliography{AS}

\end{document}